\documentclass{amsart}

\usepackage{lmodern}			
\usepackage[T1]{fontenc}
\usepackage[letterpaper,margin=0.8in]{geometry} 

\usepackage{amssymb}
\usepackage{xfrac} 
\usepackage{tikz-cd}

\usepackage[disable]{todonotes}

\usepackage[maxbibnames=99,style=alphabetic]{biblatex}
\usepackage{csquotes}
\DeclareFieldFormat{postnote}{#1} 

\usepackage{xcolor}
\usepackage{hyperref}
\usepackage{bookmark}
\definecolor{darkred}{RGB}{128,0,0}
\definecolor{darkgreen}{RGB}{0,128,0}
\definecolor{darkblue}{RGB}{0,0,128}
\hypersetup{linktocpage,
	pdfborder = {0 0 0},
	colorlinks,
	citecolor=darkgreen,
	filecolor=darkred,
	linkcolor=darkblue,
	urlcolor=cyan!50!black!90
}

\newtheorem{theorem}{Theorem}[subsection]
\newtheorem{proposition}[theorem]{Proposition}
\newtheorem{lemma}[theorem]{Lemma}

\theoremstyle{definition}
\newtheorem{definition}[theorem]{Definition}

\newtheorem{convention}[theorem]{Convention}

\newtheorem{remark}[theorem]{Remark}

\newtheorem{resources}[theorem]{Resources}

\newtheorem{acknowledgements}[theorem]{Acknowledgements}
\newtheorem{outline}[theorem]{Outline}

\usepackage{titlesec} 
\makeatletter
\renewcommand{\@secnumfont}{\bfseries} 
\renewcommand{\section}{\@startsection{section}{1}
  \z@{1.1\linespacing\@plus\linespacing} 
  {.8\linespacing} 
  {\normalfont\LARGE\bfseries}} 
\makeatother

\title{loyalty of the affine line in motivic homotopy theory}
\author{Viktor Burghardt}
\address{Simons Laufer Mathematical Sciences Institute (SLMath), 17 Gauss Way Berkeley, CA 94720-5070, USA}
\email{viktorb@u.northwestern.edu}
\urladdr{\url{https://umich.edu/~viktorb}}
\date{September 28, 2026}

\begin{document}
\maketitle

	We show that integers act invertibly on a derived scheme $S$ if and only if they act invertibly on the reduced motivic suspension spectrum of the affine line in $\mathrm{MS}_S$. As a consequence, over regular locally noetherian schemes and derived schemes on which some nonzero integer acts as zero, the motivic sphere spectrum is $\mathbf{A}^1$-invariant after inverting primes that are not invertible in the base scheme.

	\begin{outline}
		We start by briefly recalling some of the necessary notions and establishing the setting we are in. We then show that $\mathbf{A}^1$ is essentially blind to differences between the base $S$ and finite \'etale covers of $S$. This will be our main tool and it will carry us all the way to the finish line via repeated use: the affine line in motivic homotopy theory respects invertibility of integers in the base. As an immediate application, away from noninvertible integers in the base, we establish $\mathbf{A}^1$-invariance of the motivic sphere spectrum over locally noetherian regular schemes, as well as over derived schemes on which some nonzero integer acts as zero.
	\end{outline}

	\begin{resources}
		Our main sources will be \cite{AHI23}, \cite{AHI24} and \cite{Hoy24}. We will not attempt to be fully self-contained and refer the reader to these references for more background.
	\end{resources}

	\begin{acknowledgements}
		This material is partially based upon work supported by the National Science Foundation under Grant No. DMS-2424139, while the author was in residence at the Simons Laufer Mathematical Sciences Institute in Berkeley, California, during the Fall 2026 semester.
	\end{acknowledgements}

\section{set and setting}


	\begin{convention}
		Unless otherwise specified, the letter $S$ will stand for a derived scheme. We denote the $\infty$-category of nisnevich local motivic spectra over $S$ by $\mathrm{MS}_S$. We use the notation $\mathbf{A}^1=\Sigma^\infty (\mathbf{A}^1,0)\in \mathrm{MS}_S$. Given an object $E$ in a stable $\infty$-category, we will denote its rationalization by $E_\mathbf{Q}$. 
	\end{convention}

	\subsection{\texorpdfstring{$\mathbf{A}^1$}{A1}-invariance}
	We briefly recall the notion of $\mathbf{A}^1$-invariance in motivic homotopy theory and demonstrate how it relates to torsion and rational parts. 
	\begin{definition}\label{definition:a1-invariance}
		We say that $E\in \mathrm{MS}_S$ is \emph{$\mathbf{A}^1$-invariant} if the internal mapping object $\underline{\mathrm{Hom}}(\mathbf{A}^1,E)$ vanishes. We will denote the full subcategory of $\mathbf{A}^1$-invariant motivic spectra of $\mathrm{MS}_S$ by $\mathrm{Sp}(S)$.
	\end{definition}

	One way to break this condition down into smaller pieces is as follows.
	\begin{proposition}\label{proposition:a1-invariance_is_a1-invariance_rationally_and_mod_primes}
		A motitiv spectrum $E\in \mathrm{MS}_S$ is $\mathbf{A}^1$-invariant if and only if $E_\mathbf{Q}$ and $E/p$ are $\mathbf{A}^1$-invariant for all primes $p\in \mathbf{Z}$.
	\end{proposition}
	\begin{proof}
		The full subcategory $\mathrm{Sp}(S)\subset \mathrm{MS}_S$ is stable and closed under colimits by \cite[6.1]{AHI24}. If $E\in \mathrm{Sp}(S)$, then $E/p$ and $E_\mathbf{Q}$ belong to $\mathrm{Sp}(S)$ since they are colimits in $\mathrm{Sp}(S)$. On the other hand, if $E/p\in \mathrm{Sp}(S)$ for all primes $p$, then $E/n\in \mathrm{Sp}(S)$ for all integers $n$ via cofiber sequences $E/n\rightarrow E/nm \rightarrow E/m$. Moreover, if $E_\mathbf{Q}\in \mathrm{Sp}(S)$, then $E$ belongs to $\mathrm{Sp}(S)$ via the cofiber sequence $\mathrm{colim}_n \Sigma^{-1}E/n!\rightarrow E\rightarrow E_\mathbf{Q}$. 
	\end{proof}

	The following is an immediate consequence of the long exact sequence of homotopy groups.
	\begin{lemma}\label{lemma:l-squared_vanishes_on_E_mod_l_}
		Let $\mathcal{C}$ be a stable $\infty$-category, $E\in \mathcal{C}$ and $l\in \mathbf{Z}$. Then $\pi_0\mathrm{End}_\mathcal{C}(E/l)$ is a $\mathbf{Z}/l^2$-module.
	\end{lemma}

	\begin{lemma}\label{lemma:l-invertibility_is_invariance_mod_l}
		Let $\mathcal{C}$ be a stable $\infty$-category, let $A\in \mathcal{C}$ and $l\in \mathbf{Z}$. Then $l$ acts invertibly on $A$ if and only if the mapping spectrum $\mathrm{map}_\mathcal{C}(A,E/l)$ vanishes for all $E\in \mathcal{C}$.
	\end{lemma}
	\begin{proof}
		The integer $l$ acts invertibly on $\mathrm{map}_\mathcal{C}(A,E/l)$ via $A$. At the same time, $l^2$ acts as zero on it via $E/l$ by Lemma \ref{lemma:l-squared_vanishes_on_E_mod_l_}. Thus, we have $\mathrm{map}_{\mathcal{C}}(A,E/l)\simeq 0$. This proves one direction. The other direction follows by taking $E=A$.
	\end{proof}

	\begin{remark}\label{remark:l_invertible_on_a1_implies_mod_l_is_a1-invariant}
		We specialize Lemma \ref{lemma:l-invertibility_is_invariance_mod_l} to $\mathcal{C}=\mathrm{MS}_S$: $l$ acts invertibly on $\mathbf{A}^1$ in $\mathrm{MS}_S$ if and only if, for all motivic spectra $E\in \mathrm{MS}_S$, $E/l$ is $\mathbf{A}^1$-invariant.
	\end{remark}

	The following says that $\mathbf{A}^1$-invariance in $\mathrm{MS}_S$ is zariski local.
	\begin{lemma}\label{lemma:a1-invariance_is_zariski_local}
		Let $E\in \mathrm{MS}_S$. Then $\mathbf{A}^1$-invariance of $E$ can be checked on a zariski cover of $S$.
	\end{lemma}
	\begin{proof}
		Let $j: U\hookrightarrow S$ be open. By the projection formula we have $j^*(\underline{\mathrm{Hom}}(\mathbf{A}^1_S,E))\simeq\underline{\mathrm{Hom}}(\mathbf{A}^1_U,j^*(E))$. The claim now follows from zariski  of $\mathrm{MS}$.
	\end{proof}

\section{taming the affine line}
	\subsection{transport of vanishing}
	In this section we show that the affine line cannot tell the difference between $S$ and finite \'etale covers of $S$ that embed into $\mathbf{G}_m$. This is our main tool. It will allow us to investigate how functions act on $\mathbf{A}^1$ in $\mathrm{MS}_S$.

	\begin{theorem}\label{theorem:unit_of_a1_factors_through_induced}
		Let $S$ be a qcqs derived scheme. Let $Z\subset \mathbf{G}_m$ be closed finite \'etale with surjective structure map $g:Z \subset \mathbf{G}_m\rightarrow S$. Then the unit $\mathbb{S}\rightarrow \mathbf{A}^1$ factors through the full subcategory $g_\#(\mathrm{MS}_Z^\omega)$.
	\end{theorem}
	\begin{proof}
		Let $K=\mathbf{P}^1_S-Z$. Since $Z\rightarrow S$ is \'etale and separated, we may decompose $Z\times_S Z=\Delta(Z)\sqcup C$. Set $V=\mathbf{P}^1_Z-C$. Then $V\to\mathbf{P}^1_S$ is \'etale, $V\times_{\mathbf{P}^1_S}Z=\Delta(Z)\rightarrow Z$ is an isomorphism, and $V\times_{\mathbf{P}^1_S}K=V-\Delta(Z)$. Thus,
		\[\begin{tikzcd}
			V-\Delta(Z)\ar[r]\ar[d] & V\ar[d]\\
			K\ar[r,hookrightarrow] & \mathbf{P}^1_S
		\end{tikzcd}\]
		is a nisnevich square. Its top edge consists of open subschemes of $\mathbf{P}^1_Z$, hence of smooth $Z$-schemes. By nisnevich excision and exactness of $g_\#$, its fiber in $\mathrm{MS}_S$ is of the form $g_\#(W)$ for some $W\in \mathrm{MS}_Z^\omega$. 

		The automorphism $[x:y]\mapsto[x:y-x]$ of $\mathbf{P}^1$ sends $0$ to $0$ and $1$ to $\infty$. Thus, $0$ and $\infty$ are $\mathbf{P}^1$-homotopic, hence homotopic as maps $\mathbb{S}\rightarrow\mathbf{P}^1$ \cite[4.4]{AHI23}. Therefore, the map $\mathbb{S}\xrightarrow{\infty-0}K$ becomes nullhomotopic in $\mathbf{P}^1$ and factors through $g_\#(W)$.

		It now suffices to construct a map $K\rightarrow \mathbf{A}^1$ realizing the claimed factorization. Without loss of generality, we may assume that $g$ has constant rank $d$. Let $q \in \mathcal{O}(S)[s]$ be the characteristic polynomial of the coordinate on $Z$; it is monic of positive degree $d$, $q(0)$ is a unit and $V(q)=Z$. Define $\varphi:K\rightarrow \mathbf{A}^1$ by $\varphi(x)=1-\sfrac{q(0)}{q(x)}$. The map $\varphi$ is regular on $K$, with $\varphi(0)=0$ and $\varphi(\infty)=1$. It follows that the composition $\varphi\circ (\infty-0)$ is the unit of $\mathbf{A}^1$. This proves the claim.  \todo{explain somewhere that $\mathbb{S}\xrightarrow{1-0}\mathbf{A}^1$ is the unit. explain that multiplication is not addition it is multiplication. Define our notation for maps of the form $\infty-0$: they are the projections of $\infty_+-0_+$.}
	\end{proof}

	\begin{remark}\label{remark:qcqs_only_gains_compactness}
		When $S$ is not qcqs, then the proof of Theorem \ref{theorem:unit_of_a1_factors_through_induced} applies verbatim; only compactness of $W$ is lost. For our use, compactness of $W$ will not be needed.
	\end{remark}

	\begin{remark}\label{remark:transport_of_vanishing}
		A consequence of Theorem \ref{theorem:unit_of_a1_factors_through_induced} is that we can transport vanishing of maps from $\mathrm{MS}_{Z}$ to $\mathrm{MS}_S$. More precisely, let $f:E\rightarrow F$ be a map in $\mathrm{MS}_S$. Then $g^*(f)\simeq 0$ implies $f\otimes \mathrm{id}_{\mathbf{A}^1}\simeq 0$. Indeed, due to the $\mathbf{E}_\infty$-structure of $\mathbf{A}^1$ and by the projection formula, the map $f\otimes \mathrm{id}_{\mathbf{A}^1}$ has the factor $f\otimes \mathrm{id}_{g_\#(W)}\simeq g_\#(g^*(f)\otimes \mathrm{id}_W)\simeq 0$.
	\end{remark}

	\begin{remark}\label{remark:explicit_version_of_unit_of_a1_factors_through_induced}
		Let $q \in \mathcal{O}(S)[s]$ be monic of positive degree such that $q(0)$ and $\mathrm{disc}(q)$ are units. Then $Z=V(q)\subset \mathbf{G}_m$ is closed finite \'etale and surjective over $S$, and Theorem \ref{theorem:unit_of_a1_factors_through_induced} applies.
	\end{remark}

	The partial blindness of $\mathbf{A}^1$ established in Theorem \ref{theorem:unit_of_a1_factors_through_induced} implies the a priori surprising statement that $\mathbf{A}^1$-invariance comes for free in $\mathrm{MS}^-_S$.

	\begin{proposition}\label{proposition:ms_s_minus_is_a1-invariant}
		We have $\mathbf{A}^{1,-}\simeq 0$. In particular, $\mathrm{MS}^-_S=\mathrm{Sp}(S)^-$.
	\end{proposition}
	\begin{proof}
		If $2\in \mathcal{O}(S)^\times$, then Remark \ref{remark:explicit_version_of_unit_of_a1_factors_through_induced} applied with $q(x)=x^2+1$ shows that $\left<-1\right>$ acts as the identity on $\mathbf{A}^1$, since $-1$ is a square on $V(x^2+1)$ and $\left<a^2\right>=1$ for every unit $a$ \cite[5.2]{Hoy24}. If $3\in \mathcal{O}(S)^\times$, then Remark \ref{remark:explicit_version_of_unit_of_a1_factors_through_induced} applied with $q(x)=x^2+x+1$ shows that $2\left<-1\right>$ acts as $2$ on $\mathbf{A}^1$. Indeed, on $V(x^2+x+1)$ we have $-1=x+x^2$ with $x=(x^2)^2$, so that $2=\left<x\right>+\left<x^2\right>=2\left<-1\right>$ by \cite[5.3]{Hoy24}. In either case, this implies $\mathbf{A}^{1,-}\simeq 0$. In general, we cover $S$ by $S[\sfrac{1}{2}]$ and $S[\sfrac{1}{3}]$. The pullback to those opens will preserve $\mathbf{A}^{1,-}$; we know $\mathbf{A}^{1,-}$ vanishes over $S[\sfrac{1}{2}]$ and $S[\sfrac{1}{3}]$ by the above cases. The claim now follows from zariski descent of $\mathrm{MS}$.
	\end{proof}

	\subsection{loyalty of the affine line}
	The following definition captures the ability of objects to inherit invertibility in the action of primes on $S$, stably under base change.

	\begin{definition}\label{definition:loyal}
		For a prime $p \in \mathbf{Z}$, let $j_p: S[\sfrac{1}{p}]\hookrightarrow S$ be the open subscheme on which $p$ is invertible. We say $A \in \mathrm{MS}_S$ is \emph{loyal} if $p$ acts invertibly on $j_p^*(A)$ for every prime $p$. We say $A$ is \emph{fully loyal} if moreover $x^*(A/p)\neq 0$ for every prime $p\in \mathbf{Z}$ and every point $x:\mathrm{Spec}(k)\rightarrow S$ of characteristic $p$.
	\end{definition}

	\begin{remark}
		Note that $A$ is loyal if and only if, for every $f: S'\rightarrow S$ and every prime $p\in \mathbf{Z}$ invertible on $S'$, $p$ acts invertibly on $f^*(A)$. Similarly, $A$ is fully loyal if and only if, for every $f: S'\rightarrow S$ and every prime $p\in \mathbf{Z}$, $p$ acts invertibly on $f^*(A)$ exactly when $p$ is invertible on $S'$.
	\end{remark}

	\begin{theorem}\label{theorem:loyalty_of_the_affine_line}
		Let $l$ be an integer that acts invertibly on $S$. Then $l$ acts invertibly on $\mathbf{A}^1$ in $\mathrm{MS}_S$.
	\end{theorem}
	\begin{proof} 
		Let $p:\mathbf{G}_m\rightarrow S$ be the projection with coordinate $x$, and set $M=\mathbf{A}^1/l$, $M_{\mathbf{G}_m}=p^*(M)=\mathbf{A}^1_{\mathbf{G}_m}/l$. Denote $[x]:\mathbb{S}\xrightarrow{x}{\mathbf{G}_m}_+\xrightarrow{r}\Sigma^{-1}\mathbf{P}^1$ and $[x]^\natural:{\mathbf{G}_m}_+\xrightarrow{r}\Sigma^{-1}\mathbf{P}^1$. Moreover, define $q(s)=s^{l^2}-x\in \mathcal{O}(\mathbf{G}_m)[s]$, $Z=V(q)$, $g:Z\subset\mathbf{A}^1_{\mathbf{G}_m}\rightarrow \mathbf{G}_m$ with coordinate $y$, and $M_Z=g^*(M_{\mathbf{G}_m})$. We divide the proof into the following steps.
		\begin{enumerate}
			\item \emph{$\left<-1\right>$ acts as the identity on $M_{\mathbf{G}_m}$}: by Proposition \ref{proposition:ms_s_minus_is_a1-invariant} for $l$ odd, and by its proof for $l$ even.
			\item \emph{$[x]\otimes \mathrm{id}_{M_{\mathbf{G}_m}}\simeq 0$}: use Remark \ref{remark:explicit_version_of_unit_of_a1_factors_through_induced} over the base $\mathbf{G}_m$; here $q(0)=-x$ and $\mathrm{disc}(q)=\pm l^{2l^2}x^{l^2-1}$ are units because $l$ is invertible. On $Z$ we have $y^{l^2}=x$. The relations $[ab]=[a]+\left<a\right>[b]$ for units $a,b$ \cite[3.2 and 3.4]{Hoy24} give $[y^{l^2}]=\sum_{i=0}^{l^2-1}\left<y\right>^i[y]$. We claim that $\left<y\right>$ acts as the identity on $M_Z$. If $l$ is odd, then $\left<y\right>+\left<-1\right>\left<y\right>=1+\left<-1\right>$ \cite[3.5]{Hoy24}, so $2(\left<y\right>-1)$ acts as zero on $M_Z$ by $(1)$, and $2$ acts invertibly on $M_Z$ by Lemma \ref{lemma:l-squared_vanishes_on_E_mod_l_}. If $l$ is even, then $2 \in \mathcal{O}(S)^\times$, and Remark \ref{remark:explicit_version_of_unit_of_a1_factors_through_induced} over the base $Z$, applied with $q(t)=t^2-y$, together with $\left<t^2\right>=1$ \cite[5.2]{Hoy24} and Remark \ref{remark:transport_of_vanishing}, shows that $\left<y\right>-1$ acts as zero on $\mathbf{A}^1_Z$, hence on $M_Z$. Therefore $g^* ([x]\otimes \mathrm{id}_{M_{\mathbf{G}_m}})=[y^{l^2}]\otimes \mathrm{id}_{M_Z}=[y]\otimes l^2\hspace{0.5mm}\mathrm{id}_{M_Z}\simeq 0$ by Lemma \ref{lemma:l-squared_vanishes_on_E_mod_l_}. Thus, $[x]\otimes \mathrm{id}_{M_{\mathbf{G}_m}}\otimes \mathrm{id}_{\mathbf{A}^1_{\mathbf{G}_m}}\simeq 0$ by Theorem \ref{theorem:unit_of_a1_factors_through_induced}. Finally, $M_{\mathbf{G}_m}$ is a retract of $M_{\mathbf{G}_m}\otimes \mathbf{A}^1_{\mathbf{G}_m}$ through its module structure.
			\item \emph{$[x]^\natural\otimes \mathrm{id}_{M}\simeq 0$}:  $[x]\otimes \mathrm{id}_{M_{\mathbf{G}_m}}$ corresponds to $[x]^\natural\otimes \mathrm{id}_{M}$ under the adjunction $p_\#\dashv p^*$ and $p_\#p^*(M)\simeq {\mathbf{G}_m}_+\otimes M$. Conclude with $(2)$.
			\item \emph{$M\simeq 0$}: the map $[x]^\natural\otimes \mathrm{id}_{M}$ has a section by \cite[2.2]{Hoy24}. Hence, $\mathrm{id}_{\Sigma^{-1}\mathbf{P}^1\otimes M}\simeq 0$. The claim follows by invertibility of $\mathbf{P}^1$.
		\end{enumerate} 
		\vspace{-4mm}
	\end{proof}

	\begin{remark}\label{remark:a1_sees_non-invertibility}
		Let $l$ be a prime that acts noninvertibly on $S$. Then $l$ does not act invertibly on $\mathbf{A}^1$ in $\mathrm{MS}_S$. Base change to residue fields is symmetric monoidal, exact, and preserves the affine line. Thus, we may assume that $S$ is the spectrum of a characteristic $l$ field. The Hochschild homology spectrum $\mathrm{HH}$ \cite[6.4]{AHI24} now satisfies $\mathrm{HH}_0 (S)=\mathcal{O}(S)$. Thus, $l$ acts as zero on $\mathrm{HH}$. If $l$ acted invertibly on $\mathbf{A}^1$ in $\mathrm{MS}_S$, then $\mathrm{HH}/l\simeq \mathrm{HH}\oplus \Sigma \mathrm{HH}$ would be $\mathbf{A}^1$-invariant by Remark \ref{remark:l_invertible_on_a1_implies_mod_l_is_a1-invariant}. Since $\mathrm{HH}$ is clearly not $\mathbf{A}^1$-invariant, $l$ cannot act invertibly on $\mathbf{A}^1$. \todo{show that $l$ does or does not necessarily act as zero on $\mathbf{A}^1$.}
	\end{remark}

	\begin{remark}
		As a consequence of Remark \ref{remark:a1_sees_non-invertibility}, the statement of Theorem \ref{theorem:loyalty_of_the_affine_line} can be upgraded to an if and only if statement. In other words, the affine line $\mathbf{A}^1\in \mathrm{MS}_S$ is fully loyal. 
	\end{remark}

\section{one consequence}
	In this section we establish one of the many applications of loyalty of the affine line in motivic homotopy theory. Further applications, such as \'etale descent for the kernel of $\mathbf{A}^1$-localization, as well as implications for t-structures, effectivity and slices in $\mathrm{MS}_S$, will be treated elsewhere.
	\subsection{\texorpdfstring{$\mathbf{A}^1$}{A1}-invariance of the sphere}

	\begin{theorem}\label{theorem:a1-invariance_of_the_motivic_sphere_spectrum}
		Let $S$ be a regular locally noetherian scheme and let $P\subset \mathbf{Z}$ be the set of primes acting noninvertibly on $S$. Then, the $P$-inverted motivic sphere spectrum $\mathbb{S}[P^{-1}]\in \mathrm{MS}_S$ is $\mathbf{A}^1$-invariant.
	\end{theorem}
	\begin{proof}
		By Proposition \ref{proposition:a1-invariance_is_a1-invariance_rationally_and_mod_primes} it suffices to show that $\mathbb{S}[P^{-1}]$ is $\mathbf{A}^1$-invariant rationally and mod $p$, for all primes $p \in \mathbf{Z}$. Rationally, we have $\mathbb{S}_\mathbf{Q}\simeq \mathrm{H}\mathbf{Q}\oplus \mathbb{S}^-_\mathbf{Q}$, by \cite[6.3]{Hoy24}. By the assumptions on our base and Lemma \ref{lemma:a1-invariance_is_zariski_local}, $\mathrm{H}\mathbf{Q}$ is $\mathbf{A}^1$-invariant by \cite[9.12]{AHI24}; so is $\mathbb{S}^-_\mathbf{Q}$ by Proposition \ref{proposition:ms_s_minus_is_a1-invariant}. Loyalty of the affine line together with Remark \ref{remark:l_invertible_on_a1_implies_mod_l_is_a1-invariant}, implies that $\mathbb{S}[P^{-1}]/p$ is $\mathbf{A}^1$-invariant. The claim follows.
	\end{proof}

	\begin{remark}\label{remark:a1-invariance_of_the_sphere_over_q}
		The set $P$ is empty for schemes over the rational numbers $\mathbf{Q}$. In the case of regular noetherian schemes over $\mathbf{Q}$, this answers affirmatively the question, asked in \cite[1.6]{AHI24}, whether the motivic sphere over regular noetherian schemes is $\mathbf{A}^1$-invariant.
	\end{remark}

	\begin{remark}\label{remark:a1-invariance_of_the_sphere_in_pos_char}
		The generality of Theorem \ref{theorem:a1-invariance_of_the_motivic_sphere_spectrum} can be improved in various directions. For example, for $S$ a derived scheme, we have: if there exists a nonzero integer that acts as zero on $S$, then $\mathbb{S}[P^{-1}]\in \mathrm{MS}_S$ is $\mathbf{A}^1$-invariant. In particular, $S$ does not need to be regular in positive characteristic. Loyalty of the affine line reduces this claim to $\mathbf{A}^1$-invariance of $\mathrm{H}\mathbf{Q}$ and thus $\mathrm{KGL}_\mathbf{Q}$, which is $\mathbf{A}^1$-invariant by results of Weibel \cite{Wei81} and Goodwillie \cite{Goo86}.
	\end{remark}
	
	\printbibliography
\end{document}